\documentclass{amsart}
\usepackage{amssymb,latexsym}
\usepackage{amscd,amsthm}
\usepackage{tikz-cd}

\usepackage[all]{xy}

\newtheorem{theorem}{Theorem}[section]
\newtheorem{lemma}[theorem]{Lemma}
\newtheorem{proposition}[theorem]{Proposition}
\newtheorem{corollary}[theorem]{Corollary}

\theoremstyle{definition}
\newtheorem{definition}[theorem]{Definition}

\DeclareMathOperator{\Ext}{Ext}

\DeclareMathOperator{\Hom}{Hom}
\DeclareMathOperator{\Tor}{Tor}

\newcommand{\cat}[1]{\mathcal{#1}}           

\newcommand{\class}[1]{\mathcal{#1}}   

\newcommand{\Z}{\mathbb{Z}}
\newcommand{\Q}{\mathbb{Q/Z}}

\newcommand{\ch}{\textnormal{Ch}(R)}

\newcommand{\rmod}{R\textnormal{-Mod}}
\newcommand{\modr}{\textnormal{Mod-}R}

\newcommand{\rightperp}[1]{#1^{\perp}}
\newcommand{\leftperp}[1]{{}^\perp #1}

\begin{document}

\title{AC-Gorenstein rings and their stable module categories}

\author{James Gillespie}
\address{Ramapo College of New Jersey \\
         School of Theoretical and Applied Science \\
         505 Ramapo Valley Road \\
         Mahwah, NJ 07430}
\email[Jim Gillespie]{jgillesp@ramapo.edu}
\urladdr{http://pages.ramapo.edu/~jgillesp/}

\date{\today}

\begin{abstract}
We introduce what is meant by an AC-Gorenstein ring. It is a generalized notion of Gorenstein ring which is compatible with the Gorenstein AC-injective and Gorenstein AC-projective modules of Bravo-Gillespie-Hovey. It is also compatible with the notion of $n$-coherent rings introduced by Bravo-Perez:
So a $0$-coherent AC-Gorenstein ring is precisely a usual Gorenstein ring in the sense of Iwanaga, while a $1$-coherent AC-Gorenstein ring is precisely a Ding-Chen ring. We show that any AC-Gorenstein ring admits a stable module category that is compactly generated and is the homotopy category of two Quillen equivalent abelian model category structures. One is projective with cofibrant objects the Gorenstein AC-projective modules while the other is an injective model structure with fibrant objects the Gorenstein AC-injectives.
\end{abstract}

\maketitle

\section{Introduction}\label{sec-intro}

Throughout this paper we let $R$ be a ring with identity, not necessarily commutative. If $R$ is both left and right Noetherian, we say $R$ is a Gorenstein ring if it has finite injective dimension, as both a left and right module over itself. These rings were introduced and studied by Iwanaga in~\cite{iwanaga} and~\cite{iwanaga2}, generalizing (commutative) Gorenstein rings. Perhaps the most important fact about these rings is the following: Suppose the (left and right) injective dimension of $R$ is at most $d$. Then an $R$-module has finite flat dimension if and only if it has finite injective dimension, and if this is the case, each dimension can be at most $d$. We can think of this as the Fundamental Theorem of Gorenstein rings, for it leads to a very satisfying theory of Gorenstein homological algebra, which has been studied in particular by Edgar Enochs and many coauthors, but also by many other authors.

From the homotopy theoretic standpoint, the Fundamental Theorem leads to two Quillen equivalent abelian model structures on $R$-Mod, the category of left (or right) $R$-modules. This was shown by Hovey in~\cite{hovey}. Of the two model structures he constructed, one is projective in nature, with the cofibrant objects being precisely the Gorenstein projective modules from the theory of Gorenstein homological algebra, while the other is injective in its nature - the fibrant objects are the Gorenstein injective modules. But the two model structures are balanced in the sense that each have the same trivial objects. This is precisely the class of all $R$-modules satisfying the Fundamental Theorem. That is, they have finite flat (equivalently, injective) dimension, bounded by $d$. The homotopy category of these model structures provides a stable module category, denoted Stmod($R$), which is a compactly generated triangulated category.

The point of this paper is to give a generalization of Hovey's constructions described above to non-Noetherian rings. Of course we first need an appropriate notion of Gorenstein ring for non-Noetherian rings. Now for an arbitrary ring $R$, it was shown in~\cite{bravo-gillespie-hovey} that the so-called absolutely clean modules coincide with the injective modules over Noetherian rings. In fact, they even coincide with the absolutely pure modules over coherent rings, and yet over any ring they share the many important homological properties enjoyed by absolutely pure modules over coherent rings. For example, the class of absolutely clean modules is always closed under direct limits, contains all injective modules, is closed under cokernels of monomorphisms between absolutely clean modules, and there is a set (not just a proper class) of absolutely clean modules so that every absolutely clean module is built from ones in that set, as a transfinite extension.

Absolutely clean modules are defined as the modules that are injective relative to modules of type $FP_{\infty}$, where a module $F$ is said to be of \emph{type $FP_{\infty}$} if it admits a projective resolution
$$\cdots \xrightarrow{} P_2  \xrightarrow{} P_1  \xrightarrow{} P_0  \xrightarrow{} F  \xrightarrow{} 0$$ with each $P_i$ a finitely generated projective $R$-module.
So an $R$-module $A$ is called \emph{absolutely clean} if $\Ext^1_{R}(F,A)=0$ for all $R$-modules $F$ of type $FP_{\infty}$. On the other hand, we can widen our notion of flat modules too, and call a (left) $R$-module $L$ \emph{level} if $\Tor_1^R(F,L) = 0$ for all (right) $R$-modules $F$ of type $FP_{\infty}$. A ring is (right) coherent if and only if the class of level (left) modules coincide with the flat modules, but the level modules enjoy the nice homological properties that the flat modules over coherent rings enjoy. In particular, level modules are closed under products over any ring $R$.

So a natural first attempt to extend Gorenstein rings beyond the Noetherian setting would be to consider rings for which both ${}_RR$ and $R_R$ have finite absolutely clean dimension. Then imitating the proof of the Fundamental Theorem of Gorenstein rings, one begins to argue as follows: (i) Any free (left) $R$-module, being a direct sum of ${}_RR$, must also have finite absolutely clean dimension. Then (ii), any projective $R$-module, being a retract of a free module, must also have finite absolutely clean dimension. So then (iii), any flat $R$-module, being a direct limit of projective modules, must also have finite absolutely clean dimension. If the ring is coherent, you are on your way to proving the Fundamental Theorem because the level modules coincide with the flat modules over coherent rings. In fact, this generalization to coherent rings, including a statement of the Fundamental Theorem in this context, was worked out by Ding and Chen in~\cite{ding and chen 93} and~\cite{ding and chen 96}. These rings are now often called Ding-Chen rings. However, for non-coherent rings, the (left) level modules are not constructed in any reasonable way from ${}_RR$. So this first attempt described above fails. What we must do instead is to define \emph{AC-Gorenstein rings} (AC is referencing ``absolutely clean'') to be the rings for which \emph{all} (left and right) level modules, not just ${}_RR$ and $R_R$, are of finite absolutely clean dimension. We show in Section~\ref{sec-dimensions} that it is equivalent to require that all (left and right) absolutely clean modules are of finite level dimension. This leads to the main point, made in Section~\ref{sec-AC-Goren} that the Fundamental Theorem generalizes to any AC-Gorenstein ring; see Theorem~\ref{prop-dimensions}.

Indeed Theorem~\ref{prop-dimensions} describes the trivial objects for two different model structures on $R$-Mod, when $R$ is any AC-Gorenstein ring. In Section~\ref{sec-proj-model} we construct the projective model structure. Its cofibrant objects are the \emph{Gorenstein AC-projective} $R$-modules of~\cite{bravo-gillespie-hovey}; see Definition~\ref{def-Gorenstein AC-projective}. This model structure is finitely generated and so produces a compactly generated stable module category. On the other hand, in Section~\ref{sec-inj-model} we construct the dual injective model structure, with its fibrant objects being precisely the \emph{Gorenstein AC-injective} modules.

We point out that the generalization of Hovey's model structures, from the Noetherian to the coherent case, was already worked out in~\cite{gillespie-ding}. Our results agree exactly, so this paper is meant to provide the generalization in this direction beyond coherent rings. The nice thing is that our generalization is compatible with the notion of $n$-coherent rings recently studied in~\cite{bravo-perez}. We explain this connection in Section~\ref{sec-examples}, but briefly, Bravo and P\'erez classify $n$-coherent rings in terms of the absolutely clean modules. The (left) 0-coherent rings are precisely the (left) Noetherian rings while the (left) 1-coherent rings are precisely the usual (left) coherent rings. Section~\ref{sec-examples} also looks at the example of graded $R[x]/(x^2)$-modules, showing that in some cases their stable (graded) module categories are recovering $\class{D}(R)$, the derived category of the ring $R$.

\

\noindent \textbf{Acknowledgement:} The author thanks Marco P\'erez for reading this paper and providing helpful comments and feedback.

\section{Preliminaries}\label{sec-preliminaries}

Throughout the paper $R$ denotes a general ring with identity. We will work with both left and right $R$-modules.  But we will favor the left, so that $R$-module will mean left $R$-module, unless stated otherwise. The category of left $R$-modules will be denoted $\rmod$, while the category of right $R$-modules will be denoted $\modr$.

\subsection{Cotorsion pairs} Let $\cat{A}$ be an abelian category such as $R$-Mod.  By definition, a pair of classes $(\class{X},\class{Y})$ in $\cat{A}$ is called a \emph{cotorsion pair} if $\class{Y} = \rightperp{\class{X}}$ and $\class{X} = \leftperp{\class{Y}}$. Here, given a class of objects $\class{C}$ in $\cat{A}$, the right orthogonal  $\rightperp{\class{C}}$ is defined to be the class of all objects $X$ such that $\Ext^1_{\cat{A}}(C,X) = 0$ for all $C \in \class{C}$. Similarly, we define the left orthogonal $\leftperp{\class{C}}$. We call the cotorsion pair \emph{hereditary} if $\Ext^i_{\cat{A}}(X,Y) = 0$ for all $X \in \class{X}$, $Y \in \class{Y}$, and $i \geq 1$. The cotorsion pair is \emph{complete} if it has enough injectives and enough projectives. This means that for each $A \in \cat{A}$ there exist short exact sequences $0 \xrightarrow{} A \xrightarrow{} Y \xrightarrow{} X \xrightarrow{} 0$ and $0 \xrightarrow{} Y' \xrightarrow{} X' \xrightarrow{} A \xrightarrow{} 0$ with $X,X' \in \class{X}$ and $Y,Y' \in \class{Y}$.
Standard references include~\cite{enochs-jenda-book} and~\cite{trlifaj-book} and connections to abelian model categories can be found in~\cite{hovey} and~\cite{gillespie-hereditary-abelian-models}.

\subsection{Projective and injective cotorsion pairs}
Assume $\cat{A}$ is a bicomplete abelian category with enough projectives. By a \emph{projective cotorsion pair}  in $\cat{A}$ we mean a complete cotorsion pair $(\class{C},\class{W})$ for which $\class{W}$ is thick and $\class{C} \cap \class{W}$ is the class of projective objects. Such a cotorsion pair is equivalent to a \emph{projective model structure} on $\cat{A}$. By this we mean the model structure is abelian in the sense of~\cite{hovey} and all objects are fibrant. The cofibrant objects in this case are exactly those in $\class{C}$ and the trivial objects are exactly those in $\class{W}$.  We also have the dual notion of an \emph{injective cotorsion pair} $(\class{W},\class{F})$ which gives us an \emph{injective model structure} on a bicomplete abelian category with enough injectives. One could see~\cite{gillespie-recollement} for more on projective and injective cotorsion pairs. An important fact is that such cotorsion pairs are always hereditary and this implies that the associated homotopy category must be stable; that is, it is not just pre-triangulated but a triangulated category. We will use the following proposition to construct projective cotorsion pairs in this paper. It is proved in~\cite[Prop.~3.4]{bravo-gillespie-hovey}.

\begin{proposition}[Construction of a projective model
structure]\label{prop-how to create a projective model structure}
Let $\cat{A}$ be a bicomplete abelian category with enough projectives
and denote the class of projectives by $\class{P}$. Let $\class{C}$ be
any class of objects and set $\class{W} =
\rightperp{\class{C}}$. Suppose the following conditions hold:
\begin{enumerate}
\item $(\class{C},\class{W})$ is a complete cotorsion pair.
\item $\class{W}$ is thick.
\item $\class{P} \subseteq \class{W}$.
\end{enumerate}
Then there is an abelian model structure on $\cat{A}$ where every
object is fibrant, $\class{C}$ are the cofibrant objects, $\class{W}$
are the trivial objects, and $\class{P} = \class{C} \cap \class{W}$
are the trivially cofibrant objects. That is,  $(\class{C},\class{W})$ is a projective cotorsion pair.
\end{proposition}

\subsection{Finite type modules, absolutely clean, and level modules}\label{subsec-typeFP}

The following definitions were introduced in~\cite{bravo-gillespie-hovey} and are fundamental to this paper as well. Here we just give reminders of key definiteions. It is highly recommended that the reader see~\cite[Section~2]{bravo-gillespie-hovey} for further detail.

First, a module $F$ is said to be of \textbf{type $\boldsymbol{FP_{\infty}}$} if it has a projective resolution
$$\cdots \xrightarrow{} P_2  \xrightarrow{} P_1  \xrightarrow{} P_0  \xrightarrow{} F  \xrightarrow{} 0$$ with each $P_i$ finitely generated.

\begin{definition}\label{def-level}
We call an $R$-module $A$ \textbf{absolutely clean} if $\Ext^1_{R}(F,A)=0$ for all $R$-modules $F$ of type $FP_{\infty}$. On the other hand, we call an $R$-module $L$ \textbf{level} if $\Tor_1^R(F,L) = 0$ for all (right) $R$-modules $F$ of type $FP_{\infty}$.
\end{definition}

\subsection{Gorenstein AC-injective and Gorenstein AC-projective modules}
Based on the usual notions of Gorenstein injective and Gorenstein projective modules, the following definitions were introduced in~\cite{bravo-gillespie-hovey}. They may be thought of as Gorenstein injective (projective) but relative to the idea that the modules of type $FP_{\infty}$ are playing the role of ``finite module''.

\begin{definition}\label{def-Gorenstein AC-projective}
We call an $R$-module $M$ \textbf{Gorenstein AC-injective} if there exists an exact complex of injective modules $$\cdots \rightarrow I_1 \rightarrow I_0 \rightarrow I^0 \rightarrow I^1 \rightarrow \cdots$$ with $M = \ker{(I^0 \rightarrow I^1)}$ and which remains exact after applying $\Hom_{R}(A,-)$ for any absolutely clean $A$.
On the other hand, we call an $R$-module $M$ \textbf{Gorenstein AC-projective} if there exists an exact complex of projectives $$\cdots \rightarrow P_1 \rightarrow P_0 \rightarrow P^0 \rightarrow P^1 \rightarrow \cdots$$ with $M = \ker{(P^0 \rightarrow P^1)}$ and which remains exact after applying $\Hom_{R}(-,L)$ for any level module $L$.
\end{definition}

\subsection{Dimension shifting lemma}
We will make use of the following standard lemma.
\begin{lemma}[Dimension shifting]\label{lemma-dimension shifting}
Let $M$ and $N$ be $R$-modules and suppose we have exact sequences as below where the $P_i$ are projectives and the $I^i$ are injectives:

$$0 \xrightarrow{} \Omega^dM \xrightarrow{} P_{d-1} \xrightarrow{} \cdots \xrightarrow{} P_1 \xrightarrow{} P_0 \xrightarrow{} M \xrightarrow{} 0$$

$$0 \xrightarrow{} N \xrightarrow{} I^0 \xrightarrow{} I^1 \xrightarrow{} \cdots \xrightarrow{} I^{d-1} \xrightarrow{} \Sigma^dN \xrightarrow{} 0.$$

Then we have $\Ext^1_R(\Omega^dM , N) \cong \Ext^{d+1}_R(M , N) \cong \Ext^1_R(M , \Sigma^dN)$.
\end{lemma}

\begin{proof}
Using the sequence on the top and dimension shifting we get $$\Ext^{d+1}_R(M , N) \cong \Ext^{d}_R(\Omega^1M , N) \cong \Ext^{d-1}_R(\Omega^2M , N) \cong \cdots \cong \Ext^1_R(\Omega^dM , N).$$
On the other hand, using the bottom sequence and dimension shifting gives us $$\Ext^{d+1}_R(M , N) \cong \Ext^{d}_R(M , \Sigma^1N) \cong \Ext^{d-1}_R(M , \Sigma^2N) \cong \cdots \cong \Ext^1_R(M , \Sigma^dN).$$
\end{proof}

\section{Absolutely clean and level dimensions}\label{sec-dimensions}

In order to define AC-Gorenstein rings and prove the Fundamental Theorem~\ref{prop-dimensions}, we first introduce some new homological dimensions that can be attached to a ring $R$. These dimension are based on level and absolutely clean modules and we define them in this section.

\begin{definition}
Let $M$ be an $R$-module.  Its \emph{level dimension} $ld(M)$ is the smallest nonnegative integer $n$, if it exists, such that there is a resolution of $M$ by level modules $$0 \xrightarrow{}  L_n \xrightarrow{} \cdots \xrightarrow{} L_1 \xrightarrow{} L_0 \xrightarrow{} M \xrightarrow{} 0.$$
Similarly, its \emph{absolutely clean dimension} $ad(M)$ is the smallest nonnegative integer $n$, if it exists, such that there is a coresolution of $M$ by absolutely clean modules $$0 \xrightarrow{}  M \xrightarrow{} A^0 \xrightarrow{} A^1 \xrightarrow{} \cdots \xrightarrow{} A^n \xrightarrow{} 0.$$
In either case, if no such $n$ exists we set $ld(M)$ or $ad(M)$ equal to $\infty$.
\end{definition}

Of course the same definitions can be made for \emph{right} $R$-modules. Since injective modules are absolutely clean we have $ad(M) \leq id(M)$, and since flat modules are level we have $ld(M) \leq fd(M)$.

It was shown in~\cite[Section~2]{bravo-gillespie-hovey} that $\Ext^n_R(F,A) = 0$ for all $n >0$, whenever $F$ is type $FP_{\infty}$ and $A$ is absolutely clean. Using this the reader can prove the following lemma.

\begin{lemma}[absolutely clean dimension Lemma]\label{lemma-abs clean dimension lemma}
The following are equivalent:
\begin{enumerate}
\item $ad(M) \leq n$
\item $\Ext_R^{n+1}(F,M) = 0$ for all modules $F$ of type $FP_{\infty}$.
\item $\Ext_R^{n+i}(F,M) = 0$ for all $i\geq1$ and modules $F$ of type $FP_{\infty}$.
\item Whenever $0 \xrightarrow{}  M \xrightarrow{} A^0 \xrightarrow{} A^1 \xrightarrow{} \cdots \xrightarrow{} A^{n-1} \xrightarrow{} M^n \xrightarrow{} 0$ is exact with each $A^i$ absolutely clean, then $M^n$ is also absolutely clean.
\end{enumerate}
\end{lemma}

Similarly, it is shown in~\cite[Section~2]{bravo-gillespie-hovey} that $\Tor_n^R(F,L) = 0$ for all $n >0$, level modules $L$, and right modules $F$ of type $FP_{\infty}$. It yields the analogous lemma below.

\begin{lemma}[level dimension Lemma]\label{lemma-level dimension lemma}
The following are equivalent:
\begin{enumerate}
\item $ld(M) \leq n$
\item $\Tor^R_{n+1}(F,M) = 0$ for all right modules $F$ of type $FP_{\infty}$.
\item $\Tor^R_{n+i}(F,M) = 0$ for all $i\geq1$ and right modules $F$ of type $FP_{\infty}$.
\item Whenever $0 \xrightarrow{}  M_n \xrightarrow{} L_{n-1} \xrightarrow{} \cdots \xrightarrow{} L_1 \xrightarrow{} L_0 \xrightarrow{} M \xrightarrow{} 0$ is exact with each $L_i$ level, then $M_n$ is also level.
\end{enumerate}
\end{lemma}

One can use the above lemmas to easily verify the following lemma.

\begin{lemma}\label{lemma-direct summands}
We have $ad(N) \leq ad(M)$ and also $ld(N) \leq ld(M)$ whenever $N$ is a direct summand of $M$.
\end{lemma}

Recall that for any left or right $R$-module $M$, we can form the \emph{character module} $M^+ = \Hom_{\Z}(M,\Q)$. In the obvious way $M^+$ inherits a right (resp. left) $R$-module structure whenever $M$ is a left (resp. right) $R$-module.  It is shown in~\cite[Section~2]{bravo-gillespie-hovey} that $L$ is level if and only if $L^+$ is absolutely clean, and that $A$ is absolutely clean if and only if $A^+$ is level. We use this to prove the following.

\begin{proposition}\label{prop-character dimension}
Let $M$ be a left or right $R$-module and $M^+$ its character module. We have $ad(M) = ld(M^+)$, and  $ld(M) = ad(M^+)$.
\end{proposition}

\begin{proof}
We first show $ld(M^+) \leq ad(M)$. So assume $ad(M) = n < \infty$. Then by definition there is a coresolution of $M$ by absolutely clean modules $$0 \xrightarrow{}  M \xrightarrow{} A^0 \xrightarrow{} A^1 \xrightarrow{} \cdots \xrightarrow{} A^n \xrightarrow{} 0.$$ Applying $\Hom_{\Z}(-,\Q)$, we conclude that $ld(M^+) \leq n$.
On the other hand, to show $ad(M) \leq ld(M^+)$, we assume $ld(M^+) = n$, for some $n < \infty$. Then using completeness of the absolutely clean cotorsion pair, we can take a coresolution of $M$ as above where all the $A^i$ (except perhaps $A^n$) are absolutely clean. Using the absolutely clean dimension Lemma~\ref{lemma-abs clean dimension lemma} it is only left to show $A^n$ is also absolutely clean. But by applying $\Hom_{\Z}(-,\Q)$ and the level dimension Lemma~\ref{lemma-level dimension lemma}, we see that $A^{n+}$ must be level. So $A^n$ is absolutely clean.
This shows $ad(M) = ld(M^+)$, and $ld(M) = ad(M^+)$ is proved in the same way.
\end{proof}

\begin{corollary}\label{cor-double-dual-dim}
We have $ad(M) = ad(M^{++})$ and also $ld(M) = ld(M^{++})$ for any left or right $R$-module $M$.
\end{corollary}

\subsection{The global dimensions $\boldsymbol{\text{AD-lev}(R)}$ and $\boldsymbol{\text{LD-abs}(R)}$}

Recall once again that we denote the category of left $R$-modules by $R$-Mod while Mod-$R$ denotes the category of right $R$-modules.

\begin{definition}\label{def-global dims}
Let $R$ be any ring. We define the following global dimensions:
\begin{itemize}
\item $\ell \textnormal{AD-lev}(R) = \textnormal{sup}\{ad(L) \,|\, L \in \rmod \textnormal{ is level}\}$
\item $\ell \textnormal{LD-abs}(R) = \textnormal{sup}\{ld(A) \,|\, A \in \rmod \textnormal{ is absolutely clean}\}$
\item $\emph{r}\textnormal{AD-lev}(R) = \textnormal{sup}\{ad(L) \,|\, L \in \modr \textnormal{ is level}\}$
\item $\emph{r}\textnormal{LD-abs}(R) = \textnormal{sup}\{ld(A) \,|\, A \in \modr \textnormal{ is absolutely clean}\}$
\end{itemize}
We remember the notation by thinking of the first one, for example, as the \emph{\textbf{left absolutely}} clean \emph{\textbf{dimension}} among all \emph{\textbf{level}} $R$-modules.
\end{definition}

\begin{lemma}\label{lemma-global dims inequality}
Let $R$ be any ring. Then:
\begin{enumerate}
\item $\ell \textnormal{AD-lev}(R) \leq r\textnormal{LD-abs}(R)$
\item $\ell \textnormal{LD-abs}(R) \leq r\textnormal{AD-lev}(R)$
\item $r\textnormal{AD-lev}(R) \leq \ell \textnormal{LD-abs}(R)$
\item $r\textnormal{LD-abs}(R) \leq \ell \textnormal{AD-lev}(R)$
\end{enumerate}
\end{lemma}

\begin{proof}
The proofs are all similar. We just show the first one. Here we have $\ell \textnormal{AD-lev}(R)$ equals $$\textnormal{sup}\{ad(L) \,|\, L \in \rmod \textnormal{ is level}\} =  \textnormal{sup}\{ld(L^+) \,|\, L \in \rmod \textnormal{ is level}\}$$ by using Proposition~\ref{prop-character dimension}. But clearly,
$$\textnormal{sup}\{ld(L^+) \,|\, L \in \rmod \textnormal{ is level}\} \leq  \textnormal{sup}\{ld(A) \,|\, A \in \modr \textnormal{ is absolutely clean}\}.$$ So we have shown $\ell \textnormal{AD-lev}(R) \leq  r\textnormal{LD-abs}(R)$.
\end{proof}

Lemma~\ref{lemma-global dims inequality} immediately gives us the next two corollaries.

\begin{corollary}\label{cor-global dim equality}
Let $R$ be any ring. Then:
\begin{enumerate}
\item $\ell \textnormal{LD-abs}(R) = r\textnormal{AD-lev}(R)$

\item $r\textnormal{LD-abs}(R) = \ell \textnormal{AD-lev}(R)$
\end{enumerate}
\end{corollary}

For commutative rings we need not distinguish between left and right modules; we get $\ell \textnormal{AD-lev}(R) = r\textnormal{AD-lev}(R)$, which we simply denote by $\textnormal{AD-lev}(R)$. Similarly, we get $\ell \textnormal{LD-abs}(R) = r\textnormal{LD-abs}(R)$, which we denote $\textnormal{LD-abs}(R)$.
\begin{corollary}\label{cor-global dim commutative equality}
For any commutative ring $R$ we have
$\textnormal{AD-lev}(R) = \textnormal{LD-abs}(R)$.
\end{corollary}

\section{AC-Gorenstein rings}\label{sec-AC-Goren}

The purpose of this section is to define AC-Gorenstein rings and to prove the Fundamental Theorem~\ref{prop-dimensions} which characterizes their trivial $R$-modules.

\begin{definition}\label{def-AC-Goren}
We say that $R$ is an \textbf{AC-Gorenstein} ring if $$(i) \ \ \ell \textnormal{LD-abs}(R) < \infty \ \ \ \text{and} \ \ \ \emph{r}\textnormal{LD-abs}(R) < \infty.$$
Equivalently, $$(ii) \ \ell \textnormal{AD-lev}(R) < \infty \ \ \ \text{and} \ \ \  r\textnormal{AD-lev}(R) < \infty.$$
In words, (i) says there exists an integer $d < \infty$ such that any left or right absolutely clean $R$-module has level dimension at most $d$. By Corollary~\ref{cor-global dim equality} we have the  equivalent in (ii): any left or right level $R$-module has absolutely clean dimension at most $d$. By the \textbf{dimension} of such a ring we mean the least integer $d$ for which this is true.
\end{definition}

Using the results of Section~\ref{sec-dimensions}, we can now easily prove the following crucial result. It is a generalization of a well known and fundamental result for Gorenstein rings~\cite{iwanaga2} and Ding-Chen rings~\cite{ding and chen 93}.

\begin{theorem}\label{prop-dimensions}
Let $R$ be an AC-Gorenstein ring of dimension $d$. Then the following statements are equivalent for any given left or right $R$-module $M$.
\begin{enumerate}
\item $ld(M) < \infty$.
\item $ld(M) \leq d$.
\item $ad(M) < \infty$.
\item $ad(M) \leq d$.
\end{enumerate}
\end{theorem}

\begin{proof}
Note that (2) $\implies$ (1), and (4) $\implies$ (3) are trivial.

To show (1) $\implies$ (4), suppose we have an exact sequence
$$0 \xrightarrow{}  L_n \xrightarrow{} \cdots \xrightarrow{} L_1 \xrightarrow{} L_0 \xrightarrow{} M \xrightarrow{} 0$$ with each $L_i$ a level $R$-module. Since $R$ is AC-Gorenstein of dimension $d$, each $L_i$ has $ad(L_i) \leq d$. But it is easy to see, using the $\Ext_R^{d+i}$ characterization of Lemma~\ref{lemma-abs clean dimension lemma}, that the class of all modules $N$ with $ad(N) \leq d$, is a coresolving class. So we get $ad(M) \leq d$ too. This proves (1) $\implies$ (4), and (3) $\implies$ (2) can be proved similarly.
\end{proof}

\begin{definition}\label{def-trivial}
Let $R$ be an AC-Gorenstein ring of dimension $d$. We call an $R$-module \textbf{trivial} if is satisfies the equivalent conditions in Proposition~\ref{prop-dimensions}. We will denote the class of all trivial modules by $\class{W}$.
\end{definition}

The class $\class{W}$ will indeed turn out to be the class of \emph{trivial objects} in the two abelian model structures we construct on $R$-Mod. Recall that the axioms for an abelian model structure require that the class $\class{W}$ of trivial objects be a \textbf{thick} class. This means that $\class{W}$ must be closed under direct summands (retracts) and satisfy the two out of three property on short exact sequences.

\begin{corollary}\label{cor-trivial}
Let $R$ be an AC-Gorenstein ring of dimension $d$. Then the class $\class{W}$ of all trivial $R$-modules satisfies each of the following.
\begin{enumerate}
\item $\class{W}$ contains all projective and injective modules.
\item $\class{W}$ is a thick class.
\item $\class{W}$ is closed under direct products, direct sums, and in fact direct limits.
\item $\class{W}$ is closed under transfinite extensions.
\end{enumerate}
\end{corollary}

\begin{proof}
Obviously $\class{W}$ contains all projective and injective modules. Condition (2) of Lemma~\ref{lemma-abs clean dimension lemma}, or Lemma~\ref{lemma-level dimension lemma}, imply that $\class{W}$ is closed under direct summands and extensions. Now given a monomorphism $M \hookrightarrow N$ with $M, N \in \class{W}$, the induced long exact sequence in $\Ext_R^{d+i}(F,-)$ shows that its cokernel is in $\class{W}$ too. On the other hand, $\Tor^R_{d+i}(F,-)$ can be used to show  $\class{W}$ contains any kernel of any epimorphism between objects in $\class{W}$. So $\class{W}$ is a thick class.

To prove (3), use that the functor $\Ext_R^{d+1}(F,-)$ commutes with direct products while the functor $\Tor^R_{d+1}(F,-)$ commutes with direct sums and direct limits. (In fact, since $F$ is of type $FP_{\infty}$    the functor $\Tor^R_{d+1}(F,-)$ also commutes with direct products while the functor $\Ext_R^{d+1}(F,-)$ will commute with direct sums and direct limits!)

Finally, since $\class{W}$ is closed under extensions and direct limits it is closed under transfinite extensions too.
\end{proof}

\section{Examples of AC-Gorenstein rings and $n$-coherence}\label{sec-examples}

In this section we give some examples of AC-Gorenstein rings and describe a nice connection to the $n$-coherent rings recently studied by Bravo and P\'erez~\cite{bravo-perez}. We also give some particular attention to the example of graded $R[x]/(x^2)$-modules, extending some ideas from~\cite{gillespie-hovey-graded-gorenstein}.

\subsection{Iwanaga-Gorenstein rings}
Let $R$ be a (left and right) Noetherian ring. Then $R$ is AC-Gorenstein if and only if it is Gorenstein in the sense of Iwanaga~\cite{iwanaga} and~\cite{iwanaga2}. Indeed,  in the Noetherian case, level modules coincide with flat modules and absolutely clean modules coincide with injective modules.

Frobenius rings are Gorenstein and of course any Noetherian ring of finite global dimension is Gorenstein. If $R$ is a commutative Gorenstein ring and $G$ is a finite group, then the group ring $R[G]$ is Gorenstein.

\subsection{Ding-Chen rings}
Let $R$ be a (left and right) coherent ring. Then $R$ is AC-Gorenstein if and only if it is Ding-Chen in the sense of~\cite{gillespie-ding}. These are the coherent analog to Gorenstein rings and were introduced in~\cite{ding and chen 93} and~\cite{ding and chen 96}. For coherent rings, level modules coincides with flat modules and absolutely clean modules coincides with absolutely pure (FP-injective) modules.

Examples of Ding-Chen rings include all Gorenstein rings, and the group ring $R[G]$ over a finite group is Ding-Chen whenever $R$ is a commutative Ding-Chen ring. Any von-Neumann regular ring is also a Ding-Chen ring. In particular, if $R$ is an infinite product of fields, then $R$ is a Ding-Chen ring. Furthermore it follows from Theorem~7.3.1 of~\cite{glaz} that $R[x_1,x_2,x_3, \cdots , x_n]$ is a commutative Ding-Chen ring. Another example of a Ding-Chen ring is the group ring $R[G]$ where $R$ is an FC-ring (Ding-Chen of dimension 0) and $G$ is a locally finite group. See~\cite{damiano}.

\subsection{$\boldsymbol{n}$-coherent AC-Gorenstein rings}\label{subsec-n-coherent}
The notion of $n$-coherent rings were nicely studied in~\cite{bravo-perez}. Briefly, a ring $R$ is \emph{(left) $n$-coherent} ($0 \leq n \leq \infty$) if and only if the class of finitely $n$-presented (left) modules is a thick class. Here a module $F$ is \emph{finitely $n$-presented} if there exists an exact sequence
$$P_n \xrightarrow{} P_{n-1} \xrightarrow{} \cdots \xrightarrow {} P_2  \xrightarrow{} P_1  \xrightarrow{} P_0  \xrightarrow{} F  \xrightarrow{} 0$$ with each $P_i$ a finitely generated projective.
It turns out that a ring is (left) 0-coherent if and only if it is (left) Noetherian and it is (left) 1-coherent if and only if it is (left) coherent in the usual sense. As $n$ increases the $n$-coherent rings become more wild, but every ring is at least $\infty$-coherent because the modules of type $FP_{\infty}$ always form a thick class. Thus every ring can be classified as $n$-coherent for some $0 \leq n \leq \infty$.

Using this language, a Gorenstein ring is precisely a 0-coherent AC-Gorenstein ring, while a Ding-Chen ring is precisely a 1-coherent AC-Gorenstein ring. By an \textbf{$\boldsymbol{n}$-coherent AC-Gorenstein ring} we mean exactly what these words say, ``$n$-coherent'' is simply being used as an adjective.

In~\cite{Goren-n-coherent-rings}, the authors introduced a similar but different notion for some non-coherent rings, which they called \emph{Gorenstein $n$-coherent}. These rings are certainly $n$-coherent AC-Gorenstein rings  in our sense. It seems extremely unlikely however that the converse is true.

\subsection{Rings with finite weak dimension}\label{subsec-finite level dimension}
Let $R$ be a ring and $M$ be an $R$-module. Since projective modules are flat and flat modules are level, we have $ld(M) \leq fd(M) \leq pd(M)$. So any ring of finite (left and right) global dimension, or just finite weak dimension, must be AC-Gorenstein.

But of course the same is true for any ring having finite (left and right) global level dimension. As a particular example, let $k$ be a field and $R: = k[x_1,x_2,\cdots]/m^2$ be the quotient of the polynomial ring in infinitely many variables by the square of the maximal ideal $m = (x_1,x_2,\cdots)$. It is shown in~\cite[Example~2]{bravo-perez} that $R$ is 2-coherent and it follows from~\cite[Prop.~2.5]{bravo-gillespie-hovey} that every $R$-module is level. So $R$ is an example of a 2-coherent AC-Gorenstein ring of dimension 0.

\subsection{Graded $\boldsymbol{R[x]/(x^2)}$-modules}\label{subsec-chain complexes}
Let $R$ be a ring and consider the ring $A = R[x]/(x^2)$. Following ideas in~\cite{gillespie-hovey-graded-gorenstein}, we think of $A$ as a $\Z$-graded ring with a copy of $R$, generated by $1_R$, in degree 0, and another copy of $R$, generated by $x$, in degree $-1$.
We let $A$-Mod denote the category of graded left $A$-modules and Mod-$A$ the category of graded right $A$-modules. One can check that the category $A$-Mod is isomorphic to the category
$\ch$, of unbounded chain complexes of (left) $R$-modules. Indeed multiplication by $x$ corresponds to the differential $d$. Everything discussed in this paper has an analog in $A$-Mod and we have the following proposition and corollary.

\begin{proposition}\label{prop-chain complexes}
Let $X$ be a (left) $A$-module, thought of as a chain complex of (left) $R$-modules.
\begin{enumerate}
\item $X$ has finite level dimension, that is, $ld(X) \leq n < \infty$, if and only if $X$ is an exact chain complex with each cycle $Z_iX$ a (left) $R$-module of finite level dimension $ld(Z_iX) \leq n$.
\item $X$ has finite absolutely clean dimension, that is, $ad(X) \leq n < \infty$, if and only if $X$ is an exact chain complex with each cycle $Z_iX$ a (left) $R$-module of finite absolutely clean dimension $ad(Z_iX) \leq n$.
\item $X$ is Gorenstein AC-projective if and only if each component $X_n$ is Gorenstein AC-projective and any chain map $X \xrightarrow{} L$ is null homotopic whenever $L$ is a level chain complex.
\item $X$ is Gorenstein AC-injective if and only if each component $X_n$ is Gorenstein AC-injective and any chain map $A \xrightarrow{} X$ is null homotopic whenever $A$ is an absolutely clean chain complex.
\end{enumerate}
\end{proposition}

\begin{proof}
It was shown in~\cite{bravo-gillespie} that the level chain complexes are precisely the exact chain complexes whose each cycle module $Z_iX$ is level. The analogous result was also shown for the absolutely clean complexes. Now if $ld(X) \leq n < \infty$, then $X$ has a resolution by level chain complexes $$0 \xrightarrow{}  L_n \xrightarrow{} \cdots \xrightarrow{} L_1 \xrightarrow{} L_0 \xrightarrow{} X \xrightarrow{} 0.$$ It follows that $X$ must be exact and that we have an exact sequence of cycles
$$0 \xrightarrow{} Z_i(L_n) \xrightarrow{} Z_i(L_{n-1}) \xrightarrow{} \cdots \xrightarrow{}
Z_i(L_1) \xrightarrow{} Z_i(L_0) \xrightarrow{} Z_i X \xrightarrow{} 0.$$ So $ld(Z_iX) \leq n$.

It was shown in~\cite{bravo-gillespie} that the class of level modules is resolving and that level resolutions of complexes exist. So using part~(4) of the level dimension Lemma~\ref{lemma-level dimension lemma} we can also prove the converse. The case of absolutely clean dimension is similar.

Statement (4) is proved in~\cite{bravo-gillespie} and (3) is prove in~\cite{gillespie-AC-proj-complexes}.
\end{proof}

\begin{corollary}\label{cor-graded AC-Goren}
Let $R$ be a ring, and $A = R[x]/(x^2)$ the associated graded ring.
\begin{enumerate}
\item If $R$ is (left) $n$-coherent, then the graded ring $A = R[x]/(x^2)$ is also (left) $n$-coherent.

\item If $R$ is AC-Gorenstein of dimension $d$, then the graded ring $A = R[x]/(x^2)$ is also AC-Gorenstein of dimension $d$.

\item Whenever the graded ring $A$ is AC-Gorenstein of dimension $d$, the trivial (left) $A$-modules are precisely the exact (left) chain complexes $X$ having each cycle module satisfying  $ld(Z_iX) \leq d$ and/or $ad(Z_iX) \leq d$.

\item Whenever the graded ring $A$ is AC-Gorenstein of dimension $d$, there always exist nontrivial $A$-modules. That is, chain complexes $X$ for which $ld(X) = \infty$ and $ad(X) = \infty$. Indeed any non exact chain complex will do.

\item If $R$ has infinite (left and right) global level dimension, then there even exists a nontrivial $A$-module that is exact.

\end{enumerate}
\end{corollary}

\begin{proof}
(1) For $A= R[x]/(x^2)$ to be graded (left) $n$-coherent means that the class of all finitely $n$-presented (left) $R$-chain complexes is thick. But a chain complex $X$ is finitely $n$-presented if and only if it is bounded with each component $X_n$ a finitely $n$-presented (left) $R$-module~\cite[Prop.2.1.4]{Zhao-Perez-FP_n-complexes}. It follows that the graded ring $A = R[x]/(x^2)$ is (left) $n$-coherent whenever $R$ is (left) $n$-coherent.
(We also point out that the notion of (left) 0-coherent coincides with the usual notion of graded (left) Noetherian. The reason is because a homogeneous (left) ideal is precisely a graded (left) submodule of $A = R[x]/(x^2)$.)

Statements (2), (3), and (4) are consequences of Propositions~\ref{prop-chain complexes} and~\ref{prop-dimensions}.

For (5), suppose $M$ is any (left)
$R$-module with $ld(M) = ad(M) = \infty$. Then any ``disk'' complex on $M$:
$$ \cdots \xrightarrow{} 0 \xrightarrow{} 0 \xrightarrow{} M \xrightarrow{1_M} M \xrightarrow{} 0 \xrightarrow{} 0 \cdots$$ is an exact complex with a cycle of infinite  level and absolutely clean dimensions. So again using Proposition~\ref{prop-chain complexes} we conclude from (3) that this $A$-module is nontrivial.
\end{proof}

\section{The stable module category of an AC-Gorenstein ring}\label{sec-proj-model}

Let $R$ be an AC-Gorenstein ring and let $\class{W}$ be the class of all trivial $R$-modules as in Definition~\ref{def-trivial}. The goal of this section is to construct a finitely generated abelian model structure on $R$-Mod having $\class{W}$ as the class of trivial objects. This is accomplished in Theorem~\ref{them-Gorenstein AC-projectives complete cotorsion pair}. The cofibrant (resp. trivially cofibrant) objects will be the Gorenstein AC-projective (resp. categorically projective) $R$-modules.
We think of the homotopy category of this model structure as the \emph{stable module category of $R$}, and denote it by Stmod($R$), and it is a compactly generated triangulated category.

Suppose $F$ is a module of type $FP_{\infty}$. By definition this module has a projective resolution
$$\cdots \xrightarrow{} P_2  \xrightarrow{} P_1  \xrightarrow{} P_0  \xrightarrow{} F  \xrightarrow{} 0$$ with each $P_i$ finitely generated. We let $\Omega^1F = \ker(P_{0} \xrightarrow{} F)$ denote the first syzygy, and more generally $\Omega^nF = \ker(P_{n-1} \xrightarrow{} P_{n-2})$ denotes the $n$-th syzygy. Note that each $\Omega^nF$ is also of type $FP_{\infty}$.

\begin{proposition}\label{prop-ad-cot-pair}
Let $R$ be any ring and $d \geq 0$ be a given natural number.
Let $\class{AC}_d$ denote the class of all modules $N$ having $ad(N) \leq d$. Then $(\leftperp{\class{AC}_d},\class{AC}_d)$ is a cotorsion pair cogenerated by the set $\class{S} = \{\Omega^dF\}$ of all $d$-th syzygies, where $F$ ranges through the set of all (isomorphism representatives of) modules of type $FP_{\infty}$.
\end{proposition}

\begin{proof}
Certainly $(\leftperp{(\rightperp{\class{S}})},\rightperp{\class{S}})$ is a cotorsion pair cogenerated by $\class{S}$. We just need to show that $\rightperp{\class{S}} = \class{AC}_d$. We argue as in the proof of~\cite[Theorem~8.3]{hovey}.

($\subseteq$) Let $N \in \rightperp{\class{S}}$ and consider a short exact sequence
$$0 \xrightarrow{} N \xrightarrow{} I^0 \xrightarrow{} I^1 \xrightarrow{} \cdots \xrightarrow{} I^{d-1} \xrightarrow{} \Sigma^dN \xrightarrow{} 0$$ where each $I^i$ is injective. We wish to show that $\Sigma^d N$ is absolutely clean and this requires showing $\Ext^1_R(F,\Sigma^d N) = 0$ where $F$ is any module of type $FP_{\infty}$. But by Lemma~\ref{lemma-dimension shifting} we have $\Ext^1_R(F,\Sigma^dN) \cong \Ext^1_R(\Omega^dF,N)$ and this equals 0 by hypothesis.

($\supseteq$) We can reverse the above argument. Let $N \in \class{AC}_d$. Then constructing an exact sequence $$0 \xrightarrow{} N \xrightarrow{} I^0 \xrightarrow{} I^1 \xrightarrow{} \cdots \xrightarrow{} I^{d-1} \xrightarrow{} \Sigma^d N \xrightarrow{} 0$$ with each $I^i$ injective must produce a $\Sigma^dN$ that is absolutely clean, by Lemma~\ref{lemma-abs clean dimension lemma}(4). So $0 = \Ext^1_R(F,\Sigma^dN) \cong \Ext^1_R(\Omega^dF,N)$. So $N \in \rightperp{\class{S}}$.
\end{proof}

We now can prove our main theorem. Let $\class{GP}$ denote the class of all Gorenstein AC-projective modules; see Definition~\ref{def-Gorenstein AC-projective}. 

\begin{theorem}\label{them-Gorenstein AC-projectives complete cotorsion pair}
Let $R$ be an AC-Gorenstein ring of dimension $d$ and let $\class{W}$ be the class of all trivial $R$-modules as in Definition~\ref{def-trivial}.
Then $\class{M}_{prj} = (\class{GP}, \class{W})$ is a projective cotorsion pair in $R$-Mod. Furthermore, its associated model structure is finitely generated. In fact $\textnormal{Ho}(\class{M}_{prj})$, its homotopy category, is a compactly generated triangulated category with the set $\class{S} = \{\Omega^dF\}$ of Proposition~\ref{prop-ad-cot-pair} serving as a set of compact weak generators.
\end{theorem}

\begin{proof}
We give the proof in three steps.  First we show $(\leftperp{\class{W}},\class{W})$ is a projective cotorsion pair. Second, we show its homotopy category to be compactly generated by the set $\class{S} = \{\Omega^dF\}$. Last, we will show $\leftperp{\class{W}}$ is precisely the class $\class{GP}$ of Gorenstein AC-projectives.

\noindent (\textbf{Step 1}) We show that $(\leftperp{\class{W}},\class{W})$ is a projective cotorsion pair. Since, by definition, $\class{W}$ is the class of all modules in Proposition~\ref{prop-dimensions}, we get that $(\leftperp{\class{W}},\class{W}) = (\leftperp{\class{AC}_d},\class{AC}_d)$ is precisely the cotorsion pair of Proposition~\ref{prop-ad-cot-pair}. So  $(\leftperp{\class{W}},\class{W})$ is a (necessarily complete) cotorsion pair cogenerated by the set $\class{S} = \{\Omega^dF\}$. But in this case the class $\class{W}$ is thick and contains the projectives by Corollary~\ref{cor-trivial}. Hence $(\leftperp{\class{W}},\class{W})$ is a projective cotorsion pair by Proposition~\ref{prop-how to create a projective model structure}.

\noindent (\textbf{Step 2}) We now show that the set $\class{S} = \{\Omega^dF\}$ of Proposition~\ref{prop-ad-cot-pair} is a set of compact weak generators for the homotopy category associated to $\class{M} = (\leftperp{\class{W}},\class{W})$. For this, we note that each syzygy $\Omega^nF$ of a module $F$ of type $FP_{\infty}$, (with corresponding resolution $$\cdots \xrightarrow{} P_2  \xrightarrow{} P_1  \xrightarrow{} P_0  \xrightarrow{} F  \xrightarrow{} 0$$ by finitely generated projectives $P_i$), is also of type $FP_{\infty}$. So as in the proof of~\cite[Theorem~9.4]{hovey}, we get that the set $$I = \{\, \Omega^{d+1}F \hookrightarrow P_n \,\} \cup \{\,0 \hookrightarrow R\,\},$$
 provides a set of (finite) generating cofibrations for $\class{M} = (\leftperp{\class{W}},\class{W})$. Also, $J = \{\,0 \hookrightarrow R\,\}$ is a set of (finite) generating trivial cofibrations. This means the model structure is finitely generated. It now follows from a general theorem~\cite[Corollary~7.4.4]{hovey-model-categories} that the set  $\class{S} = \{\Omega^dF\}$ serves as a set of compact weak generators for the associated homotopy category  $\textnormal{Ho}(\class{M})$. In particular,  $\textnormal{Ho}(\class{M})$ is compactly generated.

\noindent (\textbf{Step 3}) We show that $M$ is Gorenstein AC-projective if and only if $M \in  \leftperp{\class{W}}$.

\noindent ($\implies$) So first suppose $M$ is Gorenstein AC-projective and let $W \in \class{W}$. We wish to show $\Ext^1_R(M,W) = 0$. Write a finite level resolution $0 \xrightarrow{} L_n \xrightarrow{} \cdots \xrightarrow{} L_1 \xrightarrow{} L_0 \xrightarrow{} W \xrightarrow{} 0$. The very definition of Gorenstein AC-projective immediately tells us that $\Ext^n_R(M,L) = 0$ for any $n>0$ and level module $L$. Using this fact we can apply a dimension shifting argument to conclude $\Ext^1_R(M,W) \cong \Ext^{n+1}_R(M,L_n) = 0$.

On the other hand suppose that $M \in \leftperp{\class{W}}$. We wish to show that $M$ is Gorenstein AC-projective. First
take a projective resolution of $M$ as below.
$$\cdots \xrightarrow{} P_2 \xrightarrow{} P_1 \xrightarrow{} P_0 \xrightarrow{} M \xrightarrow{} 0 $$ Note that since $M \in \leftperp{\class{W}}$, and $(\leftperp{\class{W}}, \class{W})$ is a hereditary cotorsion pair, the kernel at any spot in the sequence is also
in $\leftperp{\class{W}}$.
Next we use the fact that $(\leftperp{\class{W}}, \class{W})$ is a
complete cotorsion pair to find a short exact sequence $0 \xrightarrow{} M \xrightarrow{} P^0 \xrightarrow{} C \xrightarrow{} 0$
where $P^0 \in \class{W}$ and $C \in \leftperp{\class{W}}$.
But $P^0$ must also be in $\leftperp{\class{W}}$
since it is an extension of two such modules.
Since  $(\leftperp{\class{W}}, \class{W})$ has been shown to be a projective cotorsion pair it follows that $P^0$ is a projective module.
Continuing with the same procedure on $C$ we can build a projective coresolution of $M$
as below: $$0 \xrightarrow{} M \xrightarrow{} P^0 \xrightarrow{} P^1 \xrightarrow{} P^2 \xrightarrow{} \cdots.$$ Again the kernel at
each spot is in $\leftperp{\class{W}}$. Pasting this ``right'' coresolution together
with the ``left'' resolution above we get an exact
sequence $$\cdots \xrightarrow{} P_1 \xrightarrow{} P_0 \xrightarrow{} P^0 \xrightarrow{} P^1 \xrightarrow{} \cdots $$ of
projective modules which satisfies the definition of $M$ being a Gorenstein AC-projective $R$-module.
\end{proof}

\subsection{Example: Graded $\boldsymbol{R[x]/(x^2)}$-modules}
It is enlightening to examine the stable module category of the graded ring $A = R[x]/(x^2)$ from Section~\ref{subsec-chain complexes}. Proposition~\ref{prop-chain complexes}(3) characterizes the cofibrant objects. In the case that $R$ has finite (left and right) global level dimension, then an $R$-module is Gorenstein AC-projective if and only if it is projective. So then Proposition~\ref{prop-chain complexes}(3) tells us that the Gorenstein AC-projective $A$-modules are precisely the DG-projective $R$-chain complexes. Furthermore, Corollary~\ref{cor-graded AC-Goren}(3) tells us that the trivial $A$-modules in this case are precisely the exact chain complexes. It follows that the stable module category, Stmod($A$),
coincides with $\class{D}(R)$, the derived category of $R$. The model structure constructed in this section coincides with the usual projective model structure on $\ch$; see~\cite[Chapter~2]{hovey-model-categories}.

In the case that $R$ does not have finite (left and right) global level dimension,  then Stmod($A$) does not coincide with $\class{D}(R)$; because of Corollary~\ref{cor-graded AC-Goren}(5). Nevertheless, Corollary~\ref{cor-graded AC-Goren}(3) characterizes the trivial $A$-modules while Proposition~\ref{prop-chain complexes}(3) describes the cofibrant objects.

\section{The injective model structure}\label{sec-inj-model}

Again we let $R$ denote an AC-Gorenstein ring. We just constructed its stable module category, Stmod($R$), in Theorem~\ref{them-Gorenstein AC-projectives complete cotorsion pair}. It is the homotopy category of the projective cotorsion pair $\class{M}_{prj} = (\class{GP}, \class{W})$.  In this section we construct a dual (injective) model structure $\class{M}_{inj} = (\class{W}, \class{GI})$, where the class $\class{GI}$ of fibrant objects are the Gorenstein AC-injective modules of Definition~\ref{def-Gorenstein AC-projective}.

\begin{theorem}\label{them-Gorenstein AC-injectives complete cotorsion pair}
Let $R$ be an AC-Gorenstein ring of dimension $d$ and let $\class{W}$ be the class of all trivial $R$-modules as in Definition~\ref{def-trivial}.
Then $\class{M}_{inj} = (\class{W}, \class{GI})$ is an injective cotorsion pair in $R$-Mod.
\end{theorem}

The proof will use the notion of purity in the category of chain complexes. Recall a short exact sequence $\class{E} : 0 \xrightarrow{} P \xrightarrow{} X \xrightarrow{} Y \xrightarrow{} 0$ of chain complexes is called \textbf{pure} if $\Hom_{\ch}(F,\class{E})$ remains exact for any finitely presented chain complex $F$. In the same way we say $P \subseteq X$ is a \textbf{pure subcomplex} and $X/P \cong Y$ is a \textbf{pure quotient}. The proof will use the following lemma.

\begin{lemma}\label{lemma-resolutions}
Let $R$ be an AC-Gorenstein ring of dimension $d$ and let $\class{LR}_d$ denote the class of all level resolutions of length $d$. That is, the objects $L \in \class{LR}_d$ are precisely the exact chain complexes $$L \equiv 0 \xrightarrow{}  L_d \xrightarrow{} \cdots \xrightarrow{} L_1 \xrightarrow{} L_0 \xrightarrow{} W \xrightarrow{} 0$$ with each $L_i$ a level $R$-module. (Thus any such $W$ is in $\class{W}$, and conversely, for any $W \in \class{W}$ there is some $L \in \class{LR}_d$ ending in $W$.) Then the class $\class{LR}_d$ is closed under pure submodules and pure quotients.
\end{lemma}

\begin{proof}
Let $\class{E} : 0 \xrightarrow{} P \xrightarrow{} L \xrightarrow{} Y \xrightarrow{} 0$ be a pure exact sequence with $L \in \class{LR}_d$. Then~\cite[Lemma~3.2(1)]{estrada-gill-coherent} tells us that each $\class{E}_i : 0 \xrightarrow{} P_i \xrightarrow{} L_i \xrightarrow{} Y_i \xrightarrow{} 0$ is a pure exact sequence. Since the class of level modules is closed under pure submodules and pure quotients~\cite[Prop.~2.10(2)]{bravo-gillespie-hovey}, and since the class of exact complexes is also closed under pure subcomplexes and pure quotients~\cite[Lemma~3.2(4)]{estrada-gill-coherent}, we get that $P, Y \in \class{LR}_d$.
\end{proof}

\begin{proof}[Proof of Theorem~\ref{them-Gorenstein AC-injectives complete cotorsion pair}]
Let $\kappa$ be some regular cardinal $\kappa > |R|$. We let $\class{S}$ be a set (of isomorphism representatives) of all $R$-modules $W \in \class{W}$ with $|W| \leq \kappa$. Of course $(\leftperp{(\rightperp{\class{S}})},\rightperp{\class{S}})$ is a (necessarily complete) cotorsion pair cogenerated by $\class{S}$. We will now show that $\leftperp{(\rightperp{\class{S}})} = \class{W}$.

($\subseteq$)  By a well-known fact, $\leftperp{(\rightperp{\class{S}})}$ equals the class of all direct summands (retracts) of transfinite extensions of modules in the set $\class{S}$. We already know from Corollary~\ref{cor-trivial} that $\class{W}$ is closed under retracts and transfinite extensions and so we conclude $\class{S} \subseteq \leftperp{(\rightperp{\class{S}})} \subseteq \class{W}$.

($\supseteq$) Let $W \in \class{W}$. Then there exists a chain complex $$L \equiv 0 \xrightarrow{}  L_d \xrightarrow{} \cdots \xrightarrow{} L_1 \xrightarrow{} L_0 \xrightarrow{} W \xrightarrow{} 0$$ with each $L_i$ a level $R$-module. Referring to Lemma~\ref{lemma-resolutions} we see that $L \in \class{LR}_d$. As the class $\class{LR}_d$ is closed under pure subcomplexes and pure quotients we may apply~\cite[Prop.~3.4]{estrada-gill-coherent} to conclude that $L$ is a transfinite extension of complexes $L = \cup_{\alpha < \lambda} L_{\alpha}$ with each $L_0 , L_{\alpha+1}/L_{\alpha} \in \class{LR}_d$ and $|L_0| , |L_{\alpha+1}/L_{\alpha}|  \leq \kappa$. Since the definition used there for \emph{cardinality of a chain complex} is $|X| := |\coprod_{n \in \Z} X_n|$, it follows that $W$ is a transfinite extension of modules in $\class{S}$, whence $\class{W} \subseteq \leftperp{(\rightperp{\class{S}})}$.

We now have that $(\class{W},\rightperp{\class{W}})$ is a cotorsion pair cogenerated by the set $\class{S}$. Since $\class{W}$ is thick and contains all injective modules (Corollary~\ref{cor-trivial}) we conclude $(\class{W},\rightperp{\class{W}})$ is an injective cotorsion pair by the dual of Proposition~\ref{prop-how to create a projective model structure}.

Now let $M$ be an $R$-module. To finish proving the theorem we now show that $M$ is Gorenstein AC-injective if and only if $M \in \rightperp{\class{W}}$. So first suppose $M$ is Gorenstein AC-injective. Let $W \in \class{W}$. We wish to show $\Ext^1_R(W,M) = 0$. Write a finite absolutely clean coresolution $0 \xrightarrow{} W \xrightarrow{} A^0 \xrightarrow{} A^1 \xrightarrow{} \cdots \xrightarrow{} A^n \xrightarrow{} 0$. The very definition of Gorenstein AC-injective immediately tells us that $\Ext^n_R(A,M) = 0$ for any $n>0$ and absolutely clean module $A$. Using this fact we can apply a dimension shifting argument to conclude $\Ext^1_R(W,M) \cong \Ext^{n+1}_R(A^n,M) = 0$.

On the other hand suppose that $M \in \rightperp{\class{W}}$. We wish to show that $M$ is Gorenstein AC-injective. First
take an injective coresolution of $M$ as below.
$$0 \xrightarrow{} M \xrightarrow{} I^0 \xrightarrow{} I^1 \xrightarrow{} I^2 \cdots $$ Note that since $M \in \rightperp{\class{W}}$, and $(\class{W}, \rightperp{\class{W}})$ is a hereditary cotorsion pair, the kernel at any spot in the sequence is also
in $\rightperp{\class{W}}$.
Next we use the fact that $(\class{W},\rightperp{\class{W}})$ is a
complete cotorsion pair to find a short exact sequence $0 \xrightarrow{} K \xrightarrow{} I_0 \xrightarrow{} M \xrightarrow{} 0$
where $I_0 \in \class{W}$ and $K \in \rightperp{\class{W}}$.
But $I_0$ must also be in $\rightperp{\class{W}}$
since it is an extension of two such modules.
Since  $(\class{W}, \rightperp{\class{W}})$ has been shown to be an injective cotorsion pair it follows that $I^0$ is an injective module. Continuing with the same procedure on $K$ we can build an injective resolution of $M$
as below: $$ \cdots \xrightarrow{} I_1 \xrightarrow{} I_0 \xrightarrow{} M \xrightarrow{} 0.$$ Again the kernel at
each spot is in $\rightperp{\class{W}}$. Pasting this ``left'' resolution together
with the ``right'' coresolution above we get an exact
sequence $$\cdots \xrightarrow{} I_1 \xrightarrow{} I_0 \xrightarrow{} I^0 \xrightarrow{} I^1 \xrightarrow{} \cdots $$ of
injective modules with $M = \ker{(I^0 \xrightarrow{} I^1)}$. This sequence satisfies the definition of $M$ being a Gorenstein AC-injective $R$-module since now $\Hom_R(A, -)$ will leave the sequence exact for any absolutely clean module $A$.
\end{proof}

\subsection{Example: Graded $\boldsymbol{R[x]/(x^2)}$-modules}
Similar comments to the ones made at the end of Section~\ref{sec-proj-model} hold for the injective model structure. In particular, in the case that $R$ has finite (left and right) global level dimension, then an $A$-module is Gorenstein AC-injective if and only if it is a DG-injective $R$-chain complex.
So the model structure constructed in this section coincides with the usual injective model structure on $\ch$, from~\cite[Chapter~2]{hovey-model-categories}.

\end{document}